\batchmode
\makeatletter
\def\input@path{{\string"C:/Users/Connor/Dropbox/School/Research/Nonlinear Instability/Nonlinear Instability Paper/\string"/}}
\makeatother
\documentclass[11pt,english]{article}
\usepackage[T1]{fontenc}
\usepackage[latin9]{inputenc}
\usepackage{verbatim}
\usepackage{amsthm}
\usepackage{amsmath}
\usepackage{amssymb}
\usepackage{graphicx}
\usepackage{esint}

\makeatletter
\numberwithin{equation}{section}
\numberwithin{figure}{section}
\theoremstyle{plain}
\newtheorem{thm}{\protect\theoremname}
  \theoremstyle{definition}
  \newtheorem{defn}[thm]{\protect\definitionname}
  \theoremstyle{plain}
  \newtheorem{prop}[thm]{\protect\propositionname}
  \theoremstyle{plain}
  \newtheorem{lem}[thm]{\protect\lemmaname}
  \theoremstyle{remark}
  \newtheorem{rem}[thm]{\protect\remarkname}

\usepackage{lmodern}
\usepackage{fullpage}

\usepackage{amsthm}
\newtheorem{hyp}[thm]{\bf Hypothesis}

\makeatother

\usepackage{babel}
  \providecommand{\definitionname}{Definition}
  \providecommand{\lemmaname}{Lemma}
  \providecommand{\propositionname}{Proposition}
  \providecommand{\remarkname}{Remark}
\providecommand{\theoremname}{Theorem}

\begin{document}
\global\long\def\hu{H_{0}^{1}\left(U\right)}
\global\long\def\lu{L^{2}\left(U\right)}
\global\long\def\summ#1{{\displaystyle \sum_{#1=1}^{\infty}}}
\global\long\def\norm#1{\left\Vert #1\right\Vert }
\global\long\def\weak{\rightharpoonup}
\global\long\def\limm#1{{\displaystyle \lim_{#1\rightarrow\infty}}}
\global\long\def\supp#1#2{{\displaystyle \sup_{#1\in#2}}}
\global\long\def\inff#1#2{{\displaystyle \inf_{#1\in#2}}}
\global\long\def\var{\text{Var}}
\global\long\def\p#1#2{{\displaystyle \frac{\partial#1}{\partial#2}}}
\global\long\def\set#1#2{\left\{  #1\ \middle|\ #2\right\}  }
\global\long\def\pconv#1{\overset{#1}{\rightarrow}}
\global\long\def\rper{\left[-\pi,\pi\right]}
\global\long\def\mlam{\lambda_{M}}
\global\long\def\blosp{L^{2}\left(\rper;L^{2}\left(\mathbb{R}/2\pi\mathbb{Z}\right)\right)}
\global\long\def\lper{L^{2}\left(\mathbb{R}/2\pi\mathbb{Z}\right)}
\global\long\def\xirang{\left[-\frac{1}{2},\frac{1}{2}\right)}

\title{Nonlinear Instability of Periodic Traveling Waves}

\author{Connor Smith\thanks{Department of Mathematics, University of Kansas, 1460 Jayhawk Boulevard, Lawrence, KS 66045; c406s460@ku.edu.}}
\maketitle
\begin{abstract}
We study the local dynamics of $L^{2}\left(\mathbb{R}\right)$-perturbations
to the zero solution of spatially $2\pi$-periodic coefficient reaction-diffusion
systems. In this case the spectrum of the linearization about the
zero solution is purely essential and may be described via the point
spectrum of a one-parameter family of Bloch operators. When this essential
spectrum is unstable, we characterize a large class of initial perturbations
which lead to nonlinear instability of the trivial solution. This
is accomplished by using the Bloch transform to construct an appropriate
projection to capture the maximum amount of linear exponential growth
associated to the initial perturbation arising from the unstable eigenvalues
of the Bloch operators. This result is also extended to dissipative
systems of conservation laws.
\end{abstract}

\section{Introduction}

We consider a reaction-diffusion type system of real-valued differential
equations of the form

\begin{equation}
\left\{ \begin{aligned}u_{t}\left(x,t\right) & =Lu\left(x,t\right)+\mathcal{N}\left(u\left(x,t\right)\right)\\
u\left(x,0\right) & =u_{0}\left(x\right)
\end{aligned}
\right.\qquad\begin{array}{c}
x\in\mathbb{R},\ t>0\\
u\in\mathbb{R}^{d}
\end{array}\label{eq:Time Evolution Equation}
\end{equation}

posed on $X=H^{n}\left(\mathbb{R}\right)$, an appropriate Sobolev
subspace of $L^{2}\left(\mathbb{R}\right)$, where $L$ is an $2n$-th
order linear differential operator of the form

\begin{equation}
L=\sum_{j=0}^{2n}a_{j}\left(x\right)\partial_{x}^{j}\label{eq:L}
\end{equation}
with real-valued $2\pi$-periodic coefficient functions $a_{j}\left(x\right)$.
Other periods may be considered by rescaling $x$. For simplicity
we also assume that $L$ is a sectorial operator. We assume the nonlinear
operator $\mathcal{N}$ satisfies the following polynomial estimate,
\begin{equation}
\norm{\mathcal{N}\left(u\right)}_{X}\leq\norm u_{X}^{p}\qquad\text{for some }p>1.\label{eq:N Polynomial Estimate}
\end{equation}

Systems of the form \eqref{eq:Time Evolution Equation} naturally
arise when considering the local dynamics of constant coefficient
systems of reaction-diffusion systems near a spatially-periodic equilibrium
solution. That is, supposing an evolution equation $y_{t}=\mathcal{F}\left(y\right)$
has a spatially-periodic equilibrium solution $\phi$, with $\mathcal{F}\left(\phi\right)=0$,
we investigate the behavior of solutions that begin as small perturbations
of $\phi$, i.e. solutions with initial conditions of the form $y_{0}=\phi+u_{0}$.
Typically one writes the solution as $y\left(x,t\right)=\phi\left(x\right)+u\left(x,t\right)$,
substitutes this ansatz into $\mathcal{F}\left(y\right)$, then uses
the condition $\mathcal{F}\left(\phi\right)=0$ to find the ``perturbation
equations'' that govern the evolution of the perturbation $u$. In
this case equation \eqref{eq:Time Evolution Equation} would specifically
be this perturbation equation.

For the local dynamics, we use the following definition of stability,
\begin{defn}
Let $\phi$ be an equilibrium solution and $u$ be a perturbation
(as above) which satisfies \eqref{eq:Time Evolution Equation}. The
equilibrium solution $\phi$ is said to be \textit{stable} (in the
norm $\norm{\cdot}$) if for all $\epsilon>0$ there exists a $\eta>0$
so that requiring $\norm{u_{0}}<\eta$ ensures that $\norm{u\left(t\right)}<\epsilon$
for all time. Otherwise $\phi$ is said to be \textit{unstable}.
\end{defn}

Traditionally the focus has been showing that equilibrium solutions
are stable. In contrast, we are particularly interested in the case
when the $L^{2}\left(\mathbb{R}\right)$ spectrum of $L$, $\sigma\left(L\right)$,
is ``spectrally unstable'': when $\sigma\left(L\right)\cap\set{z\in\mathbb{C}}{\text{Re }z>0}\neq\emptyset$.
In \cite{henry_geometric_1993,jin_nonlinear_2019} it is shown that
spectral instability leads to instability. This is done by constructing
a specific initial perturbation $u_{0}$ which results in a poorly
behaved solution, thus precluding stability. In particular in \cite{jin_nonlinear_2019}
the initial perturbation $u_{0}$ is taken to be a perturbation which
``activates'' the most unstable part of $\sigma\left(L\right)$,
roughly speaking that it projects into the most unstable subspaces.

Where this paper differs is that we allow the initial perturbations
$u_{0}$ to be as arbitrary as possible and attempt to characterize
which could be used to preclude stability. To this end, we define
stability for an initial perturbation $u_{0}$ (in contrast to an
equilibrium solution $\phi$).

\begin{defn}
\label{def: Stability Initial Perturbation} An initial perturbation
$u_{0}$ of \eqref{eq:Time Evolution Equation} is said to be \textit{stable
}(in the norm $\norm{\cdot}$) if for all $\epsilon>0$ there exists
an $\eta>0$ so that for all $0<\delta<\eta$, if $u_{\delta}$ is
the solution to \eqref{eq:Time Evolution Equation} with initial perturbation
$\delta u_{0}$, then $\norm{u_{\delta}\left(t\right)}<\epsilon$
for all time. Otherwise $u_{0}$ is said to be \textit{unstable}.
\end{defn}

With this definition in mind, \cite{jin_nonlinear_2019} shows that
the initial perturbation which activates the rightmost part of $\sigma\left(L\right)$
is unstable. But what if an initial perturbation activates any other
part of $\sigma\left(L\right)\cap\set{z\in\mathbb{C}}{\text{Re }z>0}$?
A naive guess would be that if $u_{0}$ activates $\sigma\left(L\right)\cap\set{z\in\mathbb{C}}{\text{Re }z>0}$
it is unstable, and if it does not it is stable. If that were the
case, then one would obtain stability of $\phi$ for a wide range
of initial perturbations.

Our main result is a step in this direction with Theorem \ref{thm:Instability Theorem}
which concludes that if an initial perturbation activates an appropriate
subset of $\sigma\left(L\right)\cap\set{z\in\mathbb{C}}{\text{Re }z>0}$,
then that initial perturbation is unstable. This result applies to
many reaction-diffusion type systems with periodic equilibrium solutions,
including but not limited to scalar reaction diffusion, FitzHugh-Nagumo
\cite{rademacher_computing_2007}, the Klausmeier model for vegetation
stripe formulation \cite{sherratt_nonlinear_2007,sherratt_numerical_2013},
and the Belousov-Zhabotinskii reaction \cite{bordyugov_anomalous_2010}.
This methodology is robust enough that in Theorem \ref{thm:Applied-Problem-Instability-Theorem}
we show how it may be extended to dissipative systems of conservation
laws such as Kuramoto-Sivashinsky \cite{barker_nonlinear_2013,johnson_spectral_2015}
and the St. Venant equation \cite{barker_whitham_2010,rodrigues_periodic-coefficient_2016}.

To see how a spectral instability may lead to a nonlinear instability,
we recall from Duhamel's equation the solution of \eqref{eq:Time Evolution Equation}
can be decomposed as
\begin{equation}
u\left(x,t\right)=e^{Lt}u_{0}\left(x\right)+\int_{0}^{t}e^{L\left(t-s\right)}\mathcal{N}\left(u\left(x,s\right)\right)\ ds\,,\label{eq:Duhamel}
\end{equation}
which uses the solution semigroup $e^{Lt}$ to write the solution
in terms of a linear part and a nonlinear part.

The stability of the linear part is directly influenced by point spectrum
of $L$. Suppose $L$ had an eigenfunction $\psi$ with eigenvalue
$\lambda$ with $\text{Re }\lambda>0$: then choosing $\psi$ as an
initial perturbation, the linear part would be $e^{\lambda t}\psi$
and we would have exponential growth. While in general the $L^{2}\left(\mathbb{R}\right)$
spectrum of $L$ will not contain such eigenvalues, as the coefficient
functions $a_{j}\left(x\right)$ in \eqref{eq:L} were taken to be
$2\pi$-periodic then Floquet theory gives the $L^{2}\left(\mathbb{R}\right)$
spectrum of $L$ as the collection of $L_{\text{per}}^{2}[0,2\pi)$
point spectrum of the one-parameter family of operators $L_{\xi}=e^{-i\xi x}Le^{i\xi x}$
with $\xi\in\xirang$ \cite[Proposition 3.1]{johnson_stability_2013}.
The respective domains $L^{2}\left(\mathbb{R}\right)$ and $L_{\text{per}}^{2}[0,2\pi)$
are connected through the Bloch Transform. (This theory and its preliminaries
are developed in Subsection 2.1, and that specific spectral result
is given in Proposition \ref{prop:Bloch Spectral Decomposition}).
In Subsection 2.2 we use the Bloch Transform to define the projection
\eqref{eq:Pu} which allow us to use this unstable point spectrum
of $L_{\xi}$ to conclude exponential growth for the linear part of
\eqref{eq:Duhamel}.

This clarifies the notion of ``$u_{0}$ activating an unstable part
of $\sigma\left(L\right)$'' as ``the Bloch Transform of $u_{0}$
contains some sufficiently unstable eigenspace of some $L_{\xi}$.''
A further area of study would be to see if this Bloch transform view
gives any insight into specifically how the initial perturbation goes
unstable.

To handle the nonlinear part of \eqref{eq:Duhamel}, we apply the
reverse triangle inequality to obtain the following lower bound for
the solution,

\begin{equation}
\left|\norm{e^{Lt}u_{0}}_{X}-\norm{\int_{0}^{t}e^{L\left(t-s\right)}\mathcal{N}\left(u\left(x,s\right)\right)\ ds}_{X}\right|\leq\norm{u\left(x,t\right)}_{X}\,.\label{eq:Reverse Triangle Inequality}
\end{equation}

In Section 3 we prove an upper bound on the growth of the nonlinear
part, which when contrasted with the exponential growth of the linear
part gives instability.

\section{Spectral Properties}

\subsection{Characterization of the Spectrum}

The operator $L$ from \eqref{eq:Time Evolution Equation} is a linear
differential operator with $2\pi$-periodic coefficients, so standard
results in Floquet theory \cite[Section 2.4]{chicone_ordinary_2006}
tell us that there are no $L^{2}\left(\mathbb{R}\right)$ eigenfunctions:
the spectrum is entirely essential. Furthermore, the spectrum of $L$
can be determined from the following one-parameter family of Bloch
operators $L_{\xi}$,
\begin{equation}
L_{\xi}=e^{-i\xi x}Le^{i\xi x}\qquad\text{defined on }L_{\text{per}}^{2}[0,2\pi),\ \xi\in\xirang.\label{eq:Bloch Operator}
\end{equation}

Given the form of $L$ in \eqref{eq:L}, the Bloch operators take
the following explicit form,
\begin{equation}
L_{\xi}=\sum_{j=0}^{n}a_{j}\left(x\right)\left(\partial_{x}+i\xi\right)^{j}\,.\label{eq:L_xi}
\end{equation}

\begin{figure}[t]
\noindent \begin{centering}
\includegraphics[scale=0.8]{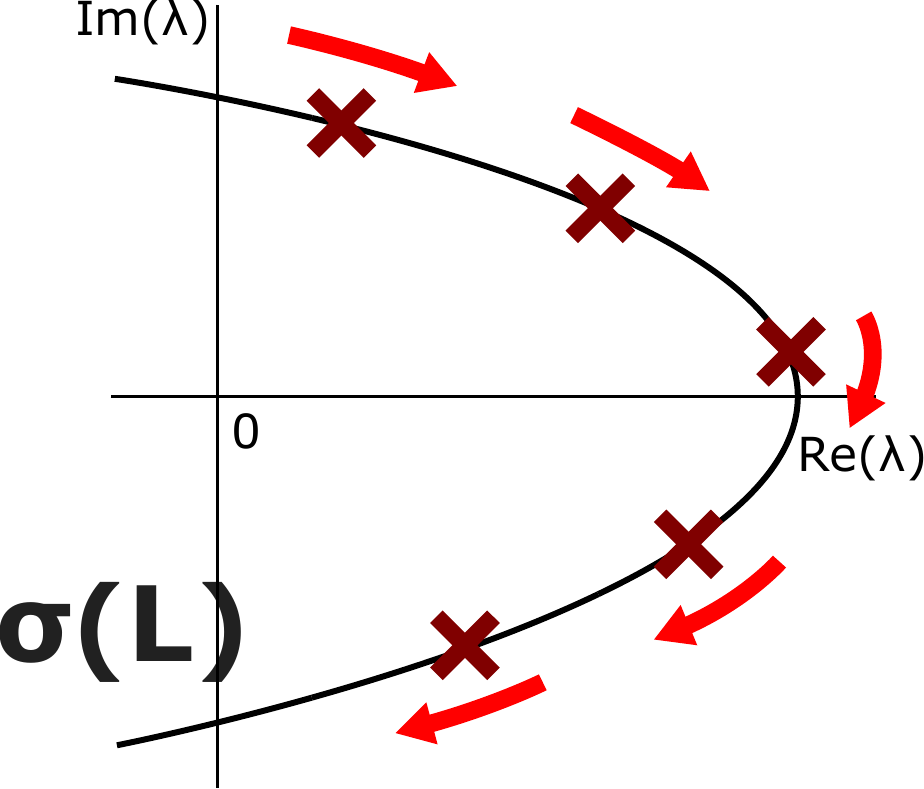}
\par\end{centering}

\caption{\label{fig:Moving-Eigenvalues} Eigenvalues $\lambda$ of $L_{\xi}$
given as maroon $x$'s. As $\xi$ changes, each individual eigenvalue
moves holomorphically with its path given in red. (It is artistic
license that the paths are  unidirectional.) The spectrum of $L$,
$\sigma\left(L\right)$, is graphed in black. The union of all of
these eigenvalues $\lambda$ forms $\sigma\left(L\right)$.}

\end{figure}

\begin{prop}
\cite[Proposition 3.1]{johnson_stability_2013} \label{prop:Bloch Spectral Decomposition}
Consider the operator $L$ as in \eqref{eq:L} acting on $L^{2}\left(\mathbb{R}\right)$
with domain $H^{1}\left(\mathbb{R}\right)$ and the associated Bloch
operators $\left\{ L_{\xi}\right\} _{\xi\in\xirang}$ acting on $L_{\text{per}}^{2}[0,2\pi)$
with domain $H_{\text{per}}^{1}[0,2\pi)$. Then $\lambda$ is in the
$L^{2}\left(\mathbb{R}\right)$ spectrum of $L$ if and only if there
exists some $\xi\in\xirang$ so that $\lambda$ is in the $L_{\text{per}}^{2}[0,2\pi)$
spectrum of $L_{\xi}$ with an eigenfunction of the form $e^{i\xi x}v\left(x\right)$
with $v\in H_{\text{per}}^{1}[0,2\pi)$.
\end{prop}

As the resolvent of each $L_{\xi}$ with domain $H_{\text{per}}^{1}[0,2\pi)$
is a compact operator, then the spectrum of $L_{\xi}$ is a countable
set of isolated eigenvalues with finite multiplicity \cite{evans_partial_1998}.
Note that from \eqref{eq:L_xi}, $\xi$ appears in $L_{\xi}$ as a
polynomial and so $L_{\xi}$ is holomorphic as a function of $\xi$,
and thus \cite[Theorem 1.7 from VII-§1]{kato_perturbation_1995} given
a closed curve $\Gamma$ that separates the spectrum, its corresponding
spectral projection is holomorphic in $\xi$ and \cite[Theorem 1.8 from VII-§1]{kato_perturbation_1995}
any finite system of eigenvalues which depend holomorphically on $\xi$.
See Figure \ref{fig:Moving-Eigenvalues} for a depiction of the spectral
picture. %

In particular, let $\lambda$ be an eigenvalue of $L_{\xi_{0}}$ and
$\Gamma$ be a curve that contains $\lambda$ and no other eigenvalue
of $L_{\xi_{0}}$. Then there is some interval $I\subset\xirang$,
with $\xi_{0}\in I$, so that the following spectral projection

\begin{equation}
\tilde{P}_{\lambda}\left(\xi\right)=\frac{1}{2\pi i}\int_{\Gamma}R\left(\zeta,L_{\xi}\right)\ d\zeta\label{eq:P-tilde}
\end{equation}
is holomorphic on $I$, where $R\left(\zeta,L_{\xi}\right)$ is the
resolvent of $L_{\xi}$.

The Bloch transform may be used to relate the domain of the Bloch
operators $\left\{ L_{\xi}\right\} _{\xi\in\xirang}$ to the domain
of $L$. To explain the former domain, first fix $\xi$ and consider
$g\left(\xi,\cdot\right)\in\mathcal{D}\left(L_{\xi}\right)=H_{\text{per}}^{1}[0,2\pi)$.
The Bloch transform requires the map $\xi\in\xirang\mapsto g\left(\xi,\cdot\right)\in H_{\text{per}}^{1}[0,2\pi)$
be $L^{2}\left(\xirang;H_{\text{per}}^{1}[0,2\pi)\right)$, identifying
this as the domain of the Bloch operators $\left\{ L_{\xi}\right\} _{\xi\in\xirang}$
. We define the Bloch transform of $f\in L^{2}\left(\mathbb{R}\right)$
to be the unique function $\check{f}\in L^{2}\left(\xirang;L_{\text{per}}^{2}[0,2\pi)\right)$
which satisfies

\begin{equation}
f\left(x\right)=\int_{-1/2}^{1/2}\check{f}\left(\xi,x\right)e^{i\xi x}\ d\xi\,.\label{eq:Bloch Transform}
\end{equation}

We have uniqueness because there is an explicit formula for $\check{f}$;
starting from the Fourier inversion formula, if we break the integral
into blocks of the form $\left[j-1/2,j+1/2\right]$ with $j\in\mathbb{Z}$,
\[
f\left(x\right)=\frac{1}{2\pi}\int_{\mathbb{R}}\hat{f}\left(\xi\right)e^{i\xi x}\ d\xi=\int_{-1/2}^{1/2}\left(\sum_{j\in\mathbb{Z}}\hat{f}\left(\xi+j\right)e^{ijx}\right)e^{i\xi x}\ d\xi\,.
\]

Then explicitly,
\[
\check{f}\left(\xi,x\right)=\sum_{j\in\mathbb{Z}}\hat{f}\left(\xi+j\right)e^{ijx}\,.
\]

From Parseval's theorem we see that the Bloch transform is an isometry,
\begin{equation}
\norm f_{L^{2}\left(\mathbb{R}\right)}^{2}=2\pi\int_{-1/2}^{1/2}\norm{\check{f}\left(\xi,\cdot\right)}_{L^{2}\left(\mathbb{R}\backslash2\pi\mathbb{Z}\right)}^{2}\ d\xi\,.\label{eq:Bloch Transform Isometry}
\end{equation}

We can also use the Bloch transform to write the linear evolution
$e^{Lt}$ in terms of the linear evolution of the Bloch operators
\cite[Equation 1.15]{johnson_nonlinear_2010}. Specifically, given
an $f_{0}\in L^{2}\left(\mathbb{R}\right),$ we have

\begin{equation}
e^{Lt}f_{0}\left(x\right)=\frac{1}{2\pi}\int_{-1/2}^{1/2}e^{i\xi x}e^{L_{\xi}t}\check{f_{0}}\left(\xi,x\right)\ d\xi\,.\label{eq: L_xi Semigroup}
\end{equation}

\subsection{Bloch-Space Projections}

Our first goal is to use the unstable spectrum of $L$ to show that
the linear part of \eqref{eq:Duhamel} has some sort of exponential
growth. We assume that the spectrum of $L$ is unstable, so let $\lambda\in\sigma\left(L\right)\cap\set{z\in\mathbb{C}}{\text{Re }z>0}$.
In Proposition \ref{prop:Bloch Spectral Decomposition} we characterized
the $L^{2}\left(\mathbb{R}\right)$ spectrum of $L$ in terms of the
$L_{\text{per}}^{2}[0,2\pi)$ eigenvalues of the one-parameter family
of Bloch operators $L_{\xi}$: there must be some $\xi_{0}\in\xirang$
so that $\lambda$ is an eigenvalue of $L_{\xi_{0}}$ with eigenfunction
$\phi\left(\xi_{0},x\right)$. Note that $\phi$ when considered as
an initial perturbation has exponential growth as $e^{L_{\xi_{0}}t}\phi=e^{\lambda t}\phi$,
albeit in $L_{\text{per}}^{2}[0,2\pi)$. We use the Bloch Transform
\eqref{eq:Bloch Transform} to extend $\phi$ into some function in
$L^{2}\left(\mathbb{R}\right)$ that has exponential growth.

First we extend $\phi\left(\xi,x\right)$ to more $\xi$ values than
just $\xi_{0}$. Using the spectral projection around a single eigenvalue
of $L_{\xi_{0}},$ $\tilde{P}_{\lambda}\left(\xi\right)$ introduced
in \eqref{eq:P-tilde}, for $\xi$ in some interval $I$ we may continuously
define $\phi\left(\xi,x\right)=\tilde{P}_{\lambda}\left(\xi\right)\phi\left(\xi_{0},x\right)$.
We restrict the contour $\Gamma$ and interval $I$ to be sufficiently
small so that 
\begin{equation}
\beta=\inf\text{Re }\Gamma>0\ ,\label{eq:Prepared Initial Data Beta}
\end{equation}

then we can use the Bloch Transform to define the following function
in $L^{2}\left(\mathbb{R}\right)$,

\begin{equation}
f_{0}\left(x\right)=\frac{1}{2\pi}\int_{I}e^{i\xi x}\phi\left(\xi,x\right)\ d\xi\ .\label{eq:O'Henry Unstable Prepared Initial Data}
\end{equation}

As each $\phi\left(\xi,x\right)$ is a linear combination of eigenfunction
of $L_{\xi}$ with eigenvalues $\lambda'$ that satisfy $\text{Re}\lambda'>\beta$,
then $\norm{e^{L_{\xi}t}\phi\left(\xi,x\right)}_{L_{\text{per}}^{2}[0,2\pi)}>e^{\beta t}\norm{\phi\left(\xi,x\right)}_{L_{\text{per}}^{2}[0,2\pi)}$.
Hence when using \eqref{eq:Bloch Transform Isometry} and \eqref{eq: L_xi Semigroup}
we have the estimate that
\begin{eqnarray*}
\norm{e^{Lt}f_{0}}_{L^{2}\left(\mathbb{R}\right)}^{2} & = & \frac{1}{2\pi}\int_{I}\norm{e^{\lambda_{i}\left(\xi\right)t}\phi_{i}\left(\xi,\cdot\right)}_{L^{2}\left(\mathbb{R}\backslash2\pi\mathbb{Z}\right)}^{2}\\
 & \geq & \frac{1}{2\pi}e^{\beta t}\inf_{\xi\in I}\norm{\phi_{i}\left(\xi,\cdot\right)}_{L^{2}\left(\mathbb{R}\backslash2\pi\mathbb{Z}\right)}^{2}\,.
\end{eqnarray*}

\begin{figure}[t]
\hfill{}\includegraphics[scale=0.8]{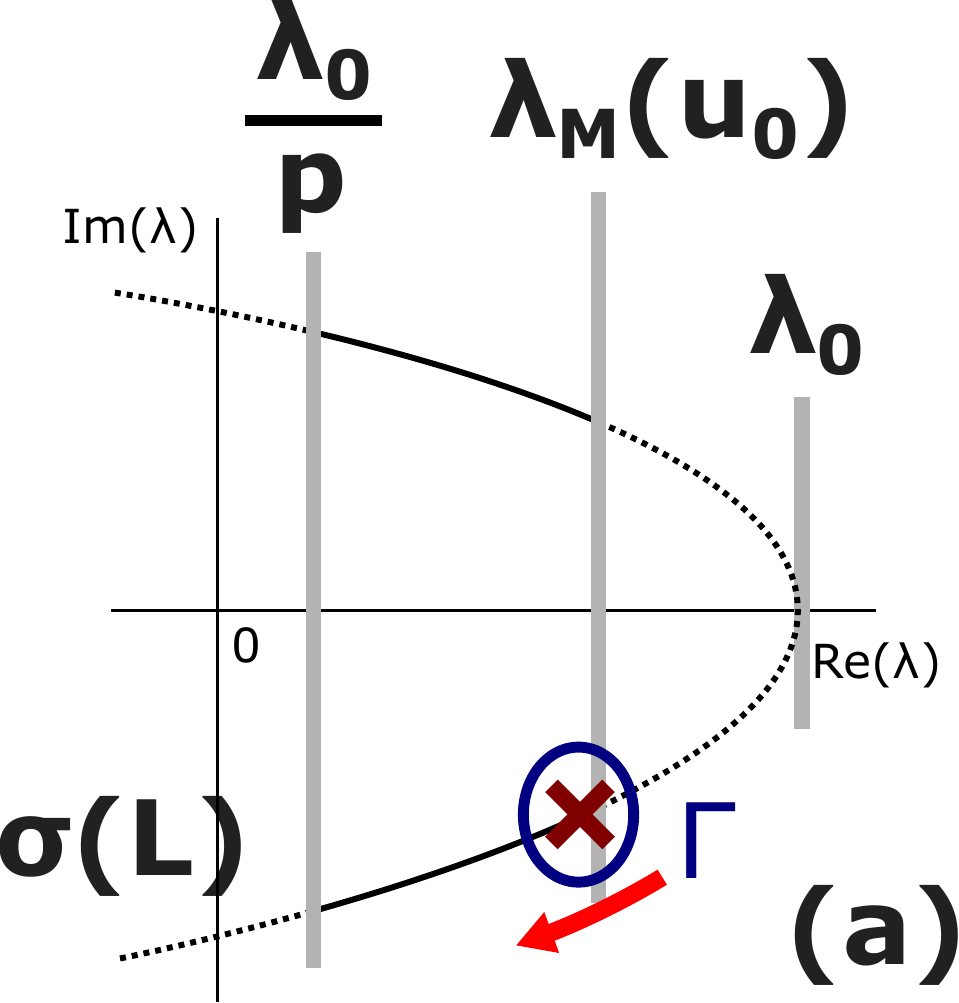}\hfill{}\includegraphics[scale=0.8]{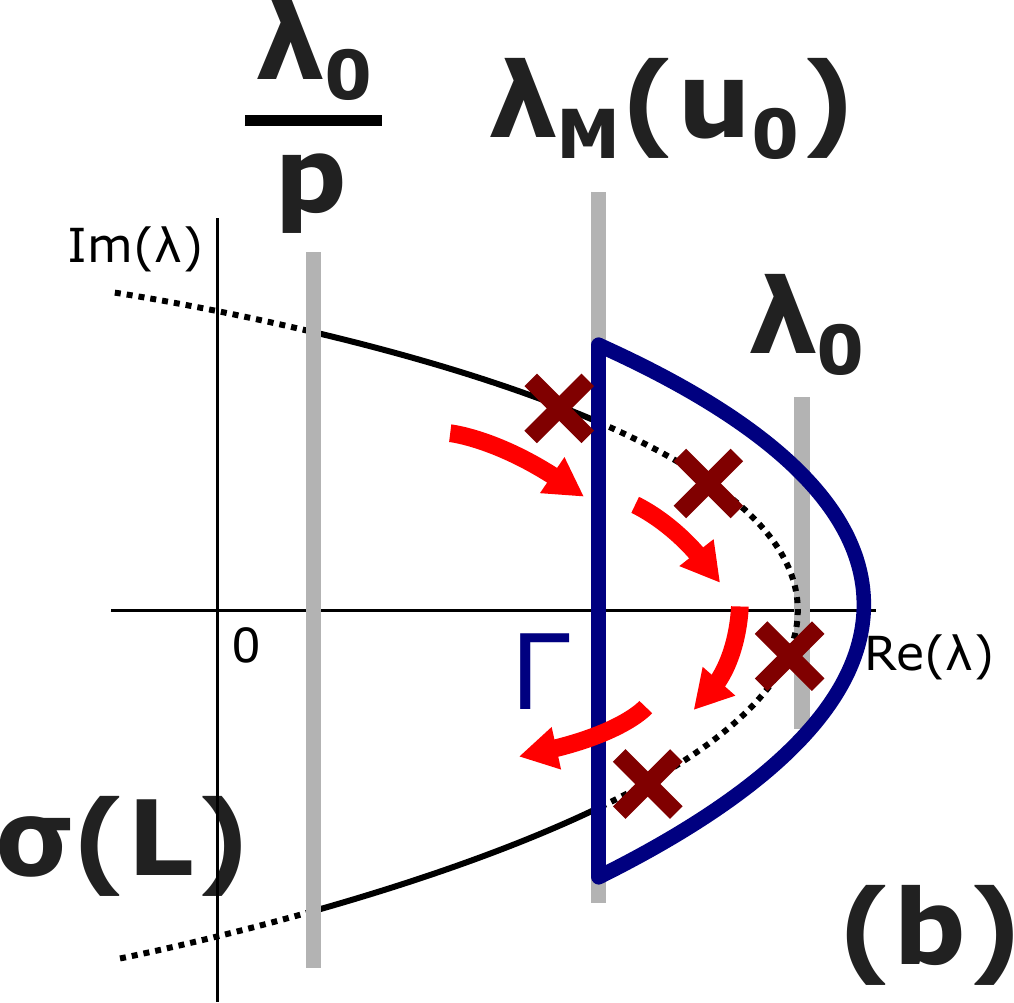}\hfill{}

\caption{\label{fig:Contours-Figure}The spectrum of $L$, $\sigma\left(L\right)$,
with the lines $\text{Re }z=\frac{\lambda_{0}}{p}$, $\text{Re }z=\lambda_{M}\left(u_{0}\right)$,
and $\text{Re }z=\lambda_{0}$ shown. (a) The contour $\Gamma$ chosen
for $P$ which is chosen close to an eigenvalue $\lambda\left(\xi_{0}\right)$,
which in turn is close to $\lambda_{M}$. The interval $I$ is chosen
sufficiently small so that no other eigenvalues enter the region enclosed
by $\Gamma$. (b) The contour $\Gamma$ chosen for $P'$. Note that
as $\xi$ varies, eigenvalues may enter or exit the region enclosed
by $\Gamma$: Hypothesis \ref{hyp: Finitely Many Indices} claims
that this happens finitely many times, so we may consider finitely
many $\xi$-intervals $I_{j}$ where the number of eigenvalues (counted
by multiplicity) enclosed by $\Gamma$ is a constant.}
\end{figure}

The intuition behind this is constructing an initial perturbation
that is close to the eigenfunction $\phi\left(\xi_{0},x\right)$.
In the sequel our strategy changes to instead defining a projection
$P$ that recognizes when such ``eigenfunctions'' are present in
an initial perturbation $u_{0}$. This has the advantage of being
widely applicable to all initial perturbations rather than just a
constructed few.

As a technical issue we require any such ``eigenfunctions'' to have
a sufficient level of exponential growth. To that end we first define
the following quantity which will be the maximum rate of exponential
growth an initial perturbation $u_{0}$ contains,

\begin{equation}
\lambda_{M}\left(u_{0}\right)=\sup\set{\text{Re }\lambda}{\tilde{P}_{\lambda}\left(u_{0}\right)^{\vee}\neq0,\lambda\in\sigma\left(L_{\xi}\right)\text{ for any }\xi\in\xirang}\,,\label{eq:Lambda_M}
\end{equation}

where $\tilde{P}_{\lambda}$ is a spectral projection to the eigenspace
of $\lambda$ as defined in equation \eqref{eq:P-tilde}. Note that
the condition in \eqref{eq:Lambda_M} is analogous to requiring ``$\lambda\in\sigma\left(L\right)$,''
but accounts for the technicality that if $\lambda$ is an eigenvalue
for multiple $L_{\xi}$ then $\tilde{P}_{\lambda}$ as defined in
\eqref{eq:P-tilde} is not unique.

We will construct this projection $P$ analogously to \eqref{eq:O'Henry Unstable Prepared Initial Data}:
as $\phi\left(\xi_{0},x\right)$ was extended to $\phi\left(\xi,x\right)$
by prepending $\tilde{P}_{\lambda}$, we shall do so here as well.
We choose our $\tilde{P}_{\lambda}$ so that $\lambda$ is arbitrarily
close to $\lambda_{M}$. That is, let $\epsilon>0$ with $0<\lambda_{M}-\epsilon<\lambda_{M}$.
Then choose an eigenvalue $\lambda$ of $L_{\xi_{0}}$ with $\tilde{P}_{\lambda}\left(u_{0}\right)^{\vee}\neq0$
so that $\lambda_{M}-\text{Re }\lambda<\frac{\epsilon}{2}$, and restrict
the contour $\Gamma$ and interval $I$ containing $\xi_{0}$ so that
$\lambda_{M}-\inf\ \text{Re }\Gamma<\epsilon$ and $\tilde{P}_{\lambda}\left(\xi\right)$
is holomorphic on $I$. See Figure \ref{fig:Contours-Figure} (a)
for an illustration. We may then use the Bloch Transform \eqref{eq:Bloch Transform}
to define the following projection,

\begin{equation}
Pu\left(x\right)=\int_{I}e^{i\xi x}\tilde{P}_{\lambda}\left(\xi\right)\check{u}\left(\xi,x\right)\ d\xi\,.\label{eq:Pu}
\end{equation}

To see that $P$ is a projection, first note that by the uniqueness
of \eqref{eq:Bloch Transform} we see that $\left(Pu\right)^{\vee}=\tilde{P}_{\lambda}\check{u}$,
and so $P^{2}u=Pu$. Secondly, applying \ref{eq:Bloch Transform Isometry}
and that $\norm{\tilde{P}_{\lambda}}_{L_{\text{per}}^{2}[0,2\pi)}\leq1$
gives that $\norm P_{L^{2}\left(\mathbb{R}\right)}\leq1$. Furthermore
the spectral projection $\tilde{P}_{\lambda}\left(\xi\right)$ commutes
with the semigroup $e^{L_{\xi}t}$ for each $\xi$ \cite[Theorem 3.14.10]{staffans_well-posed_2005},
which we can use to show that $P$ commutes with $e^{Lt}$ as well.
\begin{lem}
\label{lem:Linear Growth-1} Suppose $u_{0}$ is an initial perturbation
to \eqref{eq:Time Evolution Equation}. Then there exists some constant
$C>0$ depending only on $u_{0}$ so that for any $\omega<\lambda_{M}$
sufficiently close to $\lambda_{M}$ defined in \eqref{eq:Lambda_M},
we have the linear growth estimate
\[
C\delta e^{\omega t}\leq\norm{e^{Lt}\delta u_{0}}_{X}\,.
\]
\end{lem}
\begin{proof}
As $P$ is a projection, then
\[
\norm{Pu}_{X}\leq\norm u_{X}\,.
\]

So it suffices to show this exponential growth for the projected linear
part $Pe^{Lt}\delta u_{0}$. Set $\epsilon=\lambda_{M}-\omega$ and
recall the choice of $\lambda$ in \eqref{eq:Pu}, restricting the
interval $I$ so that $\tilde{P}_{\lambda}\left(\xi\right)\left(u_{0}\right)^{\vee}\neq0$
for all $\xi\in I$. As $\tilde{P}_{\lambda}\left(\xi\right)\left(u_{0}\right)^{\vee}$
is a linear combination of eigenfunctions of $L_{\xi}$ with eigenvalues
$\lambda'$ that satisfy $\text{Re }\lambda'>\omega$, $\norm{e^{L_{\xi}t}}_{L_{\text{per}}^{2}[0,2\pi)}>e^{\omega t}$,
and after applying \eqref{eq:Bloch Transform Isometry},
\begin{eqnarray*}
\norm{Pe^{Lt}\delta u_{0}}_{X} & \geq & \int_{I}\norm{e^{L_{\xi}t}\tilde{P}\left(\xi\right)\left(\delta u_{0}\right)^{\vee}\left(\xi,\cdot\right)\ e^{i\xi\cdot}}_{L_{\text{per}}^{2}[0,2\pi)}\ d\xi\\
 & \geq & \delta e^{\omega t}\int_{I}\inf_{\xi\in I}\left(\norm{\tilde{P}\left(\xi\right)\left(u_{0}\right)^{\vee}\left(\xi,\cdot\right)}_{L_{\text{per}}^{2}[0,2\pi)}\right)\ d\xi\,.
\end{eqnarray*}

\end{proof}

Our instability argument requires that $\lambda_{M}$ be sufficiently
large to attain a certain minimum level of exponential growth. To
define this level, we first need to determine an upper bound of $\lambda_{M}\left(u_{0}\right)$
over all choices of initial perturbations $u_{0}$, 
\begin{equation}
\lambda_{0}=\sup\set{\text{Re }\lambda}{\lambda\in\sigma\left(L\right)}\,.\label{eq:Lambda_0}
\end{equation}

Note that as $L$ was assumed to be sectorial, then $\lambda_{0}$
is necessarily finite. As part of the upcoming Hypothesis \ref{hyp: Finitely Many Indices}
we assume that this quantity is finite. We later determine in Theorem
\ref{thm:Instability Theorem} that a sufficient level of exponential
growth is attained if $\frac{\lambda_{0}}{p}<\lambda_{M}$, where
$p$ is the power of the nonlinearity as in equation \eqref{eq:N Polynomial Estimate}.
Put another way, if we define
\[
\Sigma_{U}=\sigma\left(L\right)\cap\set{\text{Re}\left(z\right)>\frac{\lambda_{0}}{p}}{z\in\mathbb{C}}
\]

then there is a sufficient amount of exponential growth if for some
$\lambda\in\Sigma_{U}$, $\lambda$ an eigenvalue of $L_{\xi_{0}}$,
we have that $\tilde{P}_{\lambda}\left(u_{0}\right)^{\vee}\neq0$.
This is what is precisely meant by saying $u_{0}$ ``activates''
the unstable part of the spectrum. We now make a hypothesis on $\Sigma_{U}$,
that eigenvalues do not enter and exit it too many times.

\begin{hyp} \label{hyp: Finitely Many Indices}

\textit{For each initial perturbation $u_{0}$ of \eqref{eq:Time Evolution Equation}
with $\frac{\lambda_{0}}{p}<\lambda_{M}\left(u_{0}\right)$, there
is a finite partition of $\xirang$ into intervals $I_{j}$ so that
the number of eigenvalues of $L_{\xi}$ (defined in \eqref{eq:L_xi},
counted by multiplicity) in}

\textit{
\[
A_{u_{0}}=\Sigma_{U}\cap\set{z\in\mathbb{C}}{\text{Re }z>\lambda_{M}\left(u_{0}\right)}
\]
 is constant for $\xi\in I_{j}$. }

\textit{}%

\textit{For a visual depiction of this latter assumption, see Figure
\ref{fig:Contours-Figure} (b), where eigenvalues are allowed to enter
and exit a contour $\Gamma$ enclosing only $A_{u_{0}}$ --- and no
other part of $\sigma\left(L\right)$ --- only finitely many times.
In \cite[Figure 6]{gani_instability_2015} \cite[Figure 3]{sherratt_numerical_2012}
a numerical calculation of the point spectrum of $L_{\xi}$ appears
to agree with this hypothesis.}

\end{hyp}

As it stands, a naive exponential growth upper bound for $u_{0}$
--- that is, obtained solely by looking at the spectrum of $L$ ---
would be $e^{\lambda_{0}t}$. If we bound the nonlinear part in \eqref{eq:Duhamel}
by this exponential function $e^{\lambda_{0}t}$, then it would overshadow
the lesser growth $e^{\left(\lambda_{M}-\epsilon\right)t}$ that \eqref{eq:Pu}
can provide for the linear part. But we can take advantage of the
fact that for all $\lambda\in\sigma\left(L\right)$ with $\text{Re }\lambda>\lambda_{M}$,
then $\tilde{P}_{\lambda}u_{0}=0$ (for any choice of $\tilde{P}_{\lambda}$).
Thus intuitively $u_{0}$ should ``ignore'' that part of the spectrum,
and the exponential growth upper bound should instead be $e^{\lambda_{M}t}$. 

\begin{lem}
\label{lem:Linear Growth-2} Suppose that \eqref{eq:Time Evolution Equation}
satisfies Hypothesis \ref{hyp: Finitely Many Indices} and let $u_{0}$
be an initial perturbation with $\frac{\lambda_{0}}{p}<\lambda_{M}\left(u_{0}\right)$.
Then there exists some constant $C>0$ depending only on $u_{0}$
so that we have the linear growth estimate
\[
\norm{e^{Lt}\delta u_{0}}_{X}\leq C\delta e^{\mlam t}\,.
\]
\end{lem}
\begin{proof}
We start by using Hypothesis \ref{hyp: Finitely Many Indices} to
find intervals $I_{j}$ a finite partition of $\xirang$ so that the
number of eigenvalues of $L_{\xi}$ (counted by multiplicity) in $A_{u_{0}}$
is constant for each $\xi\in I_{j}$. Let $\Gamma$ be a curve that
encloses all of $A_{u_{0}}$ and no other part of $\sigma\left(L\right)$
(see Figure \ref{fig:Contours-Figure} (b)). Then we define the following
spectral projection,
\[
\tilde{P}'\left(\xi\right)u=\int_{\Gamma}R\left(\zeta,L_{\xi}\right)\ d\zeta\,.
\]

Note that by construction $\left(I-\tilde{P}'\left(\xi\right)\right)\left(u_{0}\right)^{\vee}\left(\xi,x\right)=\left(u_{0}\right)^{\vee}\left(\xi,x\right)$.
Combining this with \eqref{eq: L_xi Semigroup},
\[
e^{Lt}u_{0}\left(x\right)=\frac{1}{2\pi}\sum_{j}\int_{I_{j}}e^{i\xi x}e^{L_{\xi}t}\left(I-\tilde{P}'\left(\xi\right)\right)\left(u_{0}\right)^{\vee}\left(\xi,x\right)\ d\xi
\]

and \cite[Theorem 3.14.10]{staffans_well-posed_2005} we see that
$\tilde{P}'$ commutes with $e^{L_{\xi}t}$ and $\norm{\left(I-\tilde{P}'\left(\xi\right)\right)e^{L_{\xi}t}}_{L_{\text{per}}^{2}[0,2\pi)}\leq e^{\lambda_{M}t}$.
This gives the growth estimate. 
\end{proof}

\section{Nonlinear Instability}

We now state our main instability result. With Definition \ref{def: Stability Initial Perturbation}
in mind we start by defining $u_{\delta}$, for $\delta>0$, to be
the solution to the following evolution equation,

\begin{equation}
\left\{ \begin{aligned}\left(u_{\delta}\left(x,t\right)\right)_{t} & =Lu_{\delta}\left(x,t\right)+\mathcal{N}\left(u_{\delta}\left(x,t\right)\right)\\
u_{\delta}\left(x,0\right) & =\delta u_{0}\left(x\right)\,.
\end{aligned}
\right.\label{eq:u-delta}
\end{equation}

Showing that the initial perturbation $u_{0}$ is unstable is equivalent
to showing that $u_{\delta}$ cannot be made arbitrarily small by
taking $\delta$ arbitrarily small. In our instability theorem we
find an explicit time $T$ where the solution fails to be arbitrarily
small.

\begin{thm}
\label{thm:Instability Theorem} Consider the initial value problem
\eqref{eq:Time Evolution Equation}, with $L$ a sectorial operator
with $2\pi$-periodic coefficients as defined in \eqref{eq:L}, and
$\mathcal{N}$ a nonlinear operator satisfying the polynomial estimate
\eqref{eq:N Polynomial Estimate}. Assume that Hypothesis \ref{hyp: Finitely Many Indices}
holds and let $u_{0}$ be an initial perturbation with $\lambda_{M}\left(u_{0}\right)>\lambda_{0}/p$.
Then $u_{0}$ is unstable in the sense of Definition \ref{def: Stability Initial Perturbation}
as there exist $\epsilon>0$ and $\eta>0$ sufficiently small so that
for all $\delta<\eta$, at the time $T$ when
\begin{equation}
e^{\lambda_{M}\left(u_{0}\right)T}=\frac{2\eta}{\delta}\,,\label{eq:T}
\end{equation}

we have
\[
\norm{u_{\delta}\left(\cdot,T\right)}_{L^{2}\left(\mathbb{R}\right)}>\epsilon\,,
\]
where $\lambda_{M}\left(u_{0}\right)$ is given in equation \eqref{eq:Lambda_M},
$\lambda_{0}$ is given in equation \eqref{eq:Lambda_0}, and $p$
is given in equation \eqref{eq:N Polynomial Estimate}.
\end{thm}
Recall from the introduction that our general strategy was to use
\eqref{eq:Reverse Triangle Inequality} to pit the exponential growth
of the linear term obtained from the unstable spectrum of $L$ against
the nonlinear term's slower growth. The former was developed in Lemmas
\ref{lem:Linear Growth-1} and \ref{lem:Linear Growth-2}, so we handle
the latter below. 

\begin{lem}
\label{lem:Nonlinear-Growth} For $u_{0},u_{\delta},T,\lambda_{M}\left(u_{0}\right),\lambda_{0},p$
as in Theorem \ref{thm:Instability Theorem}, then if $\lambda_{M}\left(u_{0}\right)>\lambda_{0}/p$
we have
\[
\norm{u_{\delta}\left(\cdot,t\right)}_{X}\leq\delta Ce^{\lambda_{M}\left(u_{0}\right)t}\qquad\text{for }t\leq T\,.
\]
\end{lem}
\begin{proof}
For $t\leq T$ we define the quantity
\begin{equation}
\rho\left(t\right)=\sup_{0\leq s\leq t}\norm{u_{\delta}\left(s\right)}_{X}e^{-\mlam s}\:.\label{eq:Rho}
\end{equation}

To prove the result, it is sufficient to show that $\rho\left(t\right)$
is uniformly bounded for $t\leq T$. We start by taking the norm of
\eqref{eq:Duhamel}. From Lemma \ref{lem:Linear Growth-2} we have
that $\norm{e^{Lt}\delta u_{0}}_{X}\leq\delta Ce^{\lambda_{M}t}$,
and as $L$ is sectorial%
\footnote{The assumption that $L$ is sectorial may be relaxed so long as we
have this same semigroup estimate and $\lambda_{0}$ is finite.%
} then $\norm{e^{Lt}}_{X}\leq e^{\lambda_{0}t}$. Then after recalling
\eqref{eq:N Polynomial Estimate}, we have
\[
\norm{u_{\delta}\left(\cdot,t\right)}_{X}\leq\delta Ce^{\lambda_{M}t}+\int_{0}^{t}e^{\lambda_{0}\left(t-s\right)}\norm{u_{\delta}\left(\cdot,s\right)}_{X}^{p}\ ds\,.
\]

Then we multiply and divide the nonlinear term by $e^{-p\lambda_{M}s}$,
apply \eqref{eq:Rho}, evaluate the integral, and note that $\rho\left(t\right)$
is monotone increasing to obtain

\begin{equation}
\norm{u_{\delta}\left(\cdot,t\right)}_{X}\leq\delta Ce^{\lambda_{M}t}+\rho\left(t\right)^{p}\frac{e^{p\mlam t}-e^{\lambda_{0}t}}{p\mlam-\lambda_{0}}\,.\label{eq:Nonlinear-Bound-Post-Calculation}
\end{equation}

Recall the hypothesis $\lambda_{0}/p<\lambda_{M}$, and note that
it implies
\[
e^{\lambda_{0}t}<e^{p\lambda_{M}t}\quad\text{and}\quad0<p\lambda_{M}-\lambda_{0}\,,
\]

so we may focus solely on the larger exponential growth.

Note that this upper bound \eqref{eq:Nonlinear-Bound-Post-Calculation}
also applies for all $\norm{u_{\delta}\left(\cdot,s\right)}_{X}$
for $s\leq t$. Multiplying both sides by $e^{-\lambda_{M}s}$ and
taking the supremum over all $s\leq t$ yields

\begin{equation}
\rho\left(t\right)\leq\delta C+\rho\left(t\right)^{p}\frac{e^{\left(p-1\right)\mlam t}}{p\mlam-\lambda_{0}}\,.\label{eq:Nonlinear-Bound-Polynomial}
\end{equation}

Replacing the right hand side exponential term's $t$ with $T$, recalling
\eqref{eq:T}, and dividing both sides by $\delta$,

\[
\frac{\rho\left(t\right)}{\delta}\leq C+\frac{\left(2\eta\right)^{p-1}}{p\mlam-\lambda_{0}}\left(\frac{\rho\left(t\right)}{\delta}\right)^{p}\,.
\]

Setting $z=\frac{\rho\left(t\right)}{\delta}$ leads to the equivalent
polynomial inequality valid for $t\leq T$,
\begin{equation}
0\leq C-z+\frac{\left(2\eta\right)^{p-1}}{p\mlam-\lambda_{0}}z^{p}\,.\label{eq:Nonlinear-Bound-Polynomial-2}
\end{equation}

This polynomial has a critical point at
\[
z=\frac{1}{2\eta}\left(\frac{p\mlam-\lambda_{0}}{p}\right)^{\frac{1}{p-1}}>0\,.
\]

And at this critical point the polynomial takes on the value 
\[
C-\frac{1}{2\eta}\left[\left(\frac{p\mlam-\lambda_{0}}{p}\right)^{\frac{1}{p-1}}\left(1+\frac{1}{p}\right)\right]\,.
\]

Then so long as $\eta$ is smaller than some expression that only
involves $p,\mlam,\lambda_{0}$, then the polynomial will be negative
at some $z$-value to the right of $z=0$. In particular, it will
have a root to the right of $z=0$.

When $t=0$, $z=\frac{\norm{u_{\delta}\left(0\right)}}{\delta}=1$,
so choosing $\eta$ sufficiently small will satisfy the polynomial
inequality initially at $t=0$. Then the existence of a root means
that $z$ is uniformly bounded for $t\leq T$, and hence the uniform
bound for $\rho$ for $t\leq T$.
\end{proof}
Now we can use this lemma to establish an upper bound for the nonlinear
growth in \eqref{eq:Duhamel} and finally prove Theorem \ref{thm:Instability Theorem}.
\begin{proof}
(Of Theorem \ref{thm:Instability Theorem})

Lemma \ref{lem:Linear Growth-1}, when $t=T$, gives us that
\begin{equation}
C\eta\leq\norm{e^{LT}\delta u_{0}}_{X}\,.\label{eq:Linear Part Lower Bound}
\end{equation}

From Lemma \ref{lem:Nonlinear-Growth}, we can bound the nonlinear
part of $u_{\delta}$ by

\begin{eqnarray*}
\norm{\int_{0}^{t}e^{L\left(t-s\right)}\mathcal{N}\left(u_{\delta}\left(s\right)\right)\ ds}_{X} & \leq & \left(C\delta e^{\mlam t}\right)^{p}\int_{0}^{t}e^{\mlam\left(t-s\right)}\ ds\,.
\end{eqnarray*}

In particular, when $t=T$, we have
\begin{equation}
\norm{\int_{0}^{T}e^{L\left(T-s\right)}\mathcal{N}\left(u_{\delta}\left(s\right)\right)\ ds}_{X}\leq\bar{C}\eta^{p}\,.\label{eq:Nonlinear Part Upper Bound}
\end{equation}

Then if $\eta$ is chosen sufficiently small so that $C\eta\geq\bar{C}\eta^{p}$,
then by the reverse triangle inequality we have
\[
0<C\eta-\bar{C}\eta^{p}\leq\left|\norm{e^{LT}\delta u_{0}}_{X}-\norm{\int_{0}^{T}e^{L\left(T-s\right)}\mathcal{N}\left(u_{\delta}\right)}_{X}\right|\leq\norm{u_{\delta}\left(T\right)}\,.
\]

Note that the leftmost term $C\eta-\bar{C}\eta^{p}$ is a positive
constant independent of $\delta$: this becomes our $\epsilon$, which
completes the proof.
\end{proof}

\section{Extension to Dissipative Systems of Conservation Laws}

Recall that our main Theorem  \ref{thm:Instability Theorem} was proven
in the context of reaction-diffusion type systems of the form \eqref{eq:Time Evolution Equation}:
specifically for systems with no derivatives in the nonlinearity.
Some examples of such systems would be scalar reaction diffusion,
FitzHugh-Nagumo \cite{rademacher_computing_2007}, the Klausmeier
model for vegetation stripe formulation \cite{sherratt_nonlinear_2007,sherratt_numerical_2013},
and the Belousov-Zhabotinskii reaction \cite{bordyugov_anomalous_2010}.
However, our general methodology is sufficiently robust enough that
it can apply more widely to dissipative systems of conservation laws.
{} As a specific example, consider the following Korteweg-de-Vries/Kuramoto-Sivashinsky
(KdV-KS) equation
\begin{equation}
p_{t}+pp_{x}+p_{xxx}+\beta\left(p_{xx}+p_{xxxx}\right)=0\label{eq:KdV-KS}
\end{equation}
with $0<\beta\ll1$, which arises in the context of inclined thin
film flow \cite{johnson_spectral_2015}. It was shown in \cite{barker_numerical_2014,johnson_spectral_2015}
that this equation admits periodic traveling wave solutions $\phi$
whose linearization satisfies Hypothesis \ref{hyp: Finitely Many Indices}.
If we consider solutions of the form $p\left(x,t\right)=\phi\left(x\right)+u\left(x,t\right)$,
we find that $u$ satisfies the following perturbation equation \cite[Lemma 3.3]{barker_nonlinear_2013}%
\footnote{A slight discrepancy arises in that \cite{barker_nonlinear_2013}
considers a modulation $\psi\left(x,t\right)$ so that $u\left(x,t\right)=p\left(x+\psi\left(x,t\right),t\right)-\phi\left(x\right)$
and that here we neglect such a modulation.%
}

\begin{equation}
u_{t}+u_{xxx}+\beta\left(u_{xx}+u_{xxxx}\right)+\phi u_{x}+\phi'u+uu_{x}=0\,.\label{eq:KdV Perturbation Equation}
\end{equation}

Our goal is to characterize which initial perturbations $u_{0}$ of
our traveling wave $\phi$ will result in an unstable solution $p$.
Note that the nonlinearity $uu_{x}$ does not in any standard Sobolev
space satisfy a polynomial estimate of the form \eqref{eq:N Polynomial Estimate}.
Nevertheless we can use the following damping estimate as in \cite[Proposition 3.4]{barker_nonlinear_2013}
to obtain a workable analogue. For completeness we reproduce the proof
of this damping estimate.

\begin{lem}
\label{lem:Damping Estimate} Let $u$ be a solution to \eqref{eq:KdV Perturbation Equation}.
Then $u$ satisfies the following nonlinear damping estimate
\begin{equation}
\norm u_{H^{2}}\leq e^{-\beta t}\norm{u\left(0\right)}_{H^{2}}+C\int_{0}^{T}e^{-\beta\left(t-s\right)}\norm{u\left(s\right)}_{H^{1}\left(\mathbb{R}\right)}\ ds\label{eq:Damping Estimate}
\end{equation}
for some constant $C>0$.\end{lem}
\begin{proof}
Let $\left\langle \cdot,\cdot\right\rangle $ denote the $L^{2}\left(\mathbb{R}\right)$
inner product. Using integration by parts,
\begin{equation}
\frac{1}{2}\frac{d}{dt}\left(\norm u_{H^{2}\left(\mathbb{R}\right)}^{2}\right)=\left\langle u_{t},u-u_{xx}+u_{xxxx}\right\rangle \,.\label{eq:d/dt-inner-product}
\end{equation}

We can obtain $u_{t}$ from \eqref{eq:KdV Perturbation Equation}.
Using Cauchy-Schwartz, Young's inequality, and the fact that $\norm u_{L^{\infty}\left(\mathbb{R}\right)}\leq\norm u_{H^{1}\left(\mathbb{R}\right)}$
allows us to bound the nonlinear term,
\[
\left\langle uu_{x},u_{xxxx}-u_{xx}\right\rangle \leq\frac{1}{2}\norm u_{L^{\infty}\left(\mathbb{R}\right)}\left(\left(1+\frac{1}{2\beta}\norm u_{L^{\infty}\left(\mathbb{R}\right)}\right)\norm{u_{x}}_{L^{2}\left(\mathbb{R}\right)}^{2}+\norm{u_{xx}}_{L^{2}\left(\mathbb{R}\right)}^{2}\right)+\frac{\beta}{2}\norm{u_{xxxx}}_{L^{2}\left(\mathbb{R}\right)}\,.
\]
The remainder of \eqref{eq:d/dt-inner-product} can be handled with
integration by parts and recognizing perfect derivatives, resulting
in the bound

\begin{equation}
\frac{1}{2}\frac{d}{dt}\left(\norm u_{H^{2}\left(\mathbb{R}\right)}^{2}\right)\leq\sum_{j=0}^{3}\alpha_{j}\norm{\partial_{x}^{j}u}_{L^{2}\left(\mathbb{R}\right)}^{2}-\frac{\beta}{2}\norm{u_{xxxx}}_{L^{2}\left(\mathbb{R}\right)}^{2}\,,\label{eq:Damping-Nonlinear-Part-Bound}
\end{equation}
where the $\alpha_{j}$ are non-negative constants depending on $\delta$,
$\norm{\partial_{x}^{k}\phi}_{L^{\infty}}$ for $k=0,1,2,3,4$, and
$\norm u_{L^{\infty}}$. 

Before proceeding we derive a useful Sobolev interpolation inequality.
Using integration by parts, for integer $j\geq1$ we have
\[
\norm{\partial_{x}^{j}u}_{L^{2}\left(\mathbb{R}\right)}^{2}=-\left\langle \partial_{x}^{j-1}u,\partial_{x}^{j+1}u\right\rangle \,.
\]

Using Cauchy-Schwartz and Young's inequality with an arbitrary positive
constant $a_{j}$ yields the inequality
\[
\norm{\partial_{x}^{j}u}_{L^{2}\left(\mathbb{R}\right)}^{2}\leq\frac{1}{4a_{j}}\norm{\partial_{x}^{j-1}}_{L^{2}\left(\mathbb{R}\right)}^{2}+a_{j}\norm{\partial_{x}^{j+1}u}_{L^{2}\left(\mathbb{R}\right)}^{2}\,.
\]

Taking a linear combination of the $j=2,3$ cases, for appropriate
choices of the $a_{j}$ and a sufficiently large constant $\tilde{C}>0$,
we can obtain the following Sobolev interpolation inequality

\begin{equation}
\left(\alpha_{2}+\frac{\beta}{2}\right)\norm{u_{xx}}_{L^{2}\left(\mathbb{R}\right)}^{2}+\alpha_{3}\norm{u_{xxx}}_{L^{2}\left(\mathbb{R}\right)}^{2}\leq\tilde{C}\norm{u_{x}}_{L^{2}\left(\mathbb{R}\right)}^{2}+\frac{\beta}{2}\norm{u_{xxxx}}_{L^{2}\left(\mathbb{R}\right)}^{2}\,.\label{eq:Damping-Sobolev-Interpolation}
\end{equation}

Combining \eqref{eq:Damping-Sobolev-Interpolation} and \eqref{eq:Damping-Nonlinear-Part-Bound}
gives, for some constant $C>0$,
\[
\frac{1}{2}\frac{d}{dt}\left(\norm u_{H^{2}\left(\mathbb{R}\right)}^{2}\right)\leq-\frac{\beta}{2}\norm u_{H^{2}\left(\mathbb{R}\right)}^{2}+C\norm u_{H^{1}\left(\mathbb{R}\right)}^{2}\,.
\]
The estimate \eqref{eq:Damping Estimate} then follows from Gronwall's
inequality.
\end{proof}
The key ingredient to the damping estimate was that the leading term
$-\beta u_{xxxx}$ of \eqref{eq:KdV Perturbation Equation} was negative:
such an estimate is not necessarily limited to the specific PDE \eqref{eq:KdV-KS}.
In particular a damping estimate also holds for the St. Venant equation
\cite[Proposition 4.2]{rodrigues_periodic-coefficient_2016}, which
is a hyperbolic-parabolic system of balance laws. In light of this,
we introduce a general requirement on the nonlinearity, that for constants
$\theta>0$, $C>0$, $T>0$, $n\geq2$, $p>1$, and $t<T$,

\begin{equation}
\norm{\mathcal{N}\left(u\left(t\right)\right)}_{H^{n}\left(\mathbb{R}\right)}\leq\norm{u\left(t\right)}_{H^{1}\left(\mathbb{R}\right)}^{p-1}\left(e^{-\theta t}\norm{u\left(0\right)}_{H^{n}\left(\mathbb{R}\right)}+C\int_{0}^{T}e^{-\theta\left(t-s\right)}\norm{u\left(s\right)}_{H^{1}\left(\mathbb{R}\right)}\ ds\right)\,.\label{eq:Damping Nonlinearity Bound}
\end{equation}

We can then prove an alternate instability result.

\begin{thm}
\label{thm:Applied-Problem-Instability-Theorem} Consider the initial
value problem \eqref{eq:Time Evolution Equation}, with $L$ a sectorial
operator with $2\pi$-periodic coefficients as defined in \eqref{eq:L},
and $\mathcal{N}$ a nonlinear operator instead satisfying the estimate
\eqref{eq:Damping Nonlinearity Bound}. Assume that Hypothesis \ref{hyp: Finitely Many Indices}
holds and let $u_{0}$ be an initial perturbation which satisfies
both $\frac{\lambda_{0}+\theta}{p-1}<\lambda_{M}\left(u_{0}\right)$
and $\frac{\lambda_{0}}{p}<\lambda_{M}\left(u_{0}\right)$, then $u_{0}$
is unstable in the sense of Definition \ref{def: Stability Initial Perturbation}
as there exist $\epsilon>0$ and $\eta>0$ sufficiently small so that
for all $\delta<\eta$, at the time $T$ when
\begin{equation}
e^{\lambda_{M}\left(u_{0}\right)T}=\frac{2\eta}{\delta}\,,\label{eq:T-1-1}
\end{equation}

we have
\[
\norm{u_{\delta}\left(\cdot,T\right)}_{L^{2}\left(\mathbb{R}\right)}>\epsilon\,,
\]
where $\lambda_{M}\left(u_{0}\right)$ is given in equation \eqref{eq:Lambda_M},
$\lambda_{0}$ is given in equation \eqref{eq:Lambda_0}, and $u_{\delta}$
is the solution to \eqref{eq:u-delta}.\end{thm}
\begin{rem}
In Lemma \ref{lem:Damping Estimate} we show that \eqref{eq:KdV Perturbation Equation}
satisfies a nonlinear estimate of the form \eqref{eq:Damping Nonlinearity Bound}
with $\theta=\beta$. However, equation \eqref{eq:Damping-Sobolev-Interpolation}
of the proof can be modified to obtain any $0<\theta<\beta$, so the
$\theta$ in the requirement that $\frac{\lambda_{0}+\theta}{p-1}<\lambda_{M}$
is not restrictive. \end{rem}
\begin{proof}
Here we take $X=H^{1}\left(\mathbb{R}\right)$. Lemmas \ref{lem:Linear Growth-1}
and \ref{lem:Linear Growth-2} and the proof of Theorem \ref{thm:Instability Theorem}
apply with no further modification, provided we can establish an estimate
as in Lemma \ref{lem:Nonlinear-Growth}.

We again define $\rho$ as in \eqref{eq:Rho} and start by concentrating
on the nonlinear term of \eqref{eq:Duhamel}. By the triangle inequality
and \eqref{eq:Damping Nonlinearity Bound},

\begin{eqnarray*}
\int_{0}^{t}\norm{e^{L\left(t-s\right)}}_{X}\norm{\mathcal{N}\left(u\left(s\right)\right)}_{X}\ ds & \leq & \int_{0}^{t}e^{\lambda_{0}\left(t-s\right)}\norm{u\left(s\right)}_{X}^{p-1}e^{-\theta s}\norm{u\left(0\right)}_{H^{n}\left(\mathbb{R}\right)}\ ds\\
 &  & +C\int_{0}^{t}e^{\lambda_{0}\left(t-s\right)}\norm{u\left(s\right)}_{X}^{p-1}\int_{0}^{s}e^{-\theta\left(s-s'\right)}\norm{u\left(s'\right)}_{X}\ ds'\ ds\,.
\end{eqnarray*}

We multiply and divide by appropriate powers of $e^{\lambda_{M}s}$,
bound by $\rho$, and evaluate the integrals to obtain the analogue
of \eqref{eq:Nonlinear-Bound-Post-Calculation},

\begin{eqnarray}
\norm{u_{\delta}\left(t\right)}_{X} & \leq & \delta Ce^{\lambda_{M}t}+\rho^{p-1}\left(t\right)\norm{u_{\delta}\left(0\right)}_{H^{n}}\frac{e^{\left(\left(p-1\right)\lambda_{M}-\theta\right)t}-e^{\lambda_{0}t}}{\left(p-1\right)\lambda_{M}-\left(\theta+\lambda_{0}\right)}\label{eq:Damping-Nonlinear-Growth}\\
 &  & +\left(\frac{C\rho^{p}\left(t\right)}{\theta+\lambda_{M}}\right)\left(\frac{e^{p\lambda_{M}t}-e^{\lambda_{0}t}}{p\lambda_{M}-\lambda_{0}}-\frac{e^{\left(\left(p-1\right)\lambda_{M}-\theta\right)t}-e^{\lambda_{0}t}}{\left(p-1\right)\lambda_{M}-\left(\theta+\lambda_{0}\right)}\right)\,.\nonumber 
\end{eqnarray}

Note that the hypotheses $\frac{\lambda_{0}+\theta}{p-1}<\lambda_{M}$
and $\frac{\lambda_{0}}{p}<\lambda_{M}$ imply that
\[
e^{\lambda_{0}t}<e^{\left(\left(p-1\right)\lambda_{M}-\theta\right)t}\quad\text{and}\quad e^{\lambda_{0}t}<e^{p\lambda_{M}t}\,,
\]

so we may focus solely on the larger exponential growth.

The upper bound in \eqref{eq:Damping-Nonlinear-Growth} also applies
for all $\norm{u_{\delta}\left(s\right)}_{X}$ for $s\leq t$. Then
after multiplying both sides of the inequality by $e^{-\lambda_{M}t}$
and taking the supremum over all $s\leq t$, we have the analogue
of \eqref{eq:Nonlinear-Bound-Polynomial},

\[
\rho\left(t\right)\leq\delta C+a_{p-1}\,\rho^{p-1}\left(t\right)e^{\left(p-2\right)\lambda_{M}t}+a_{p}\,\rho^{p}\left(t\right)e^{\left(p-1\right)\lambda_{M}t}\,,
\]

where $a_{p-1}$ and $a_{p}$ depend only on $C$, $p$, $\lambda_{0}$,
$\lambda_{M}$, $\theta$, and $\norm{u_{\delta}\left(0\right)}_{H^{n}}$.

Replacing the right hand side exponential terms' $t$ with $T$, recalling
\eqref{eq:T-1-1}, dividing both sides by $\delta$ and setting $z=\frac{\rho\left(t\right)}{\delta}$,
then we have the analogue of \eqref{eq:Nonlinear-Bound-Polynomial-2},
\begin{equation}
0\leq a_{p}\,\eta^{p-1}z^{p}+a_{p-1}\,\eta^{p-2}\delta\, z^{p-1}-z+C\,.\label{eq:Applied Example Polynomial Bound}
\end{equation}

To find a zero of this polynomial, we compare it with the linear function
$C-z$. In particular, on the interval $\left[0,L\right]$ we have
\[
\left|\left(a_{p}\,\eta^{p-1}z^{p}+a_{p-1}\,\eta^{p-2}\delta\, z^{p-1}-z+C\right)-\left(C-z\right)\right|\leq a_{p}\,\eta^{p-1}L^{p}+a_{p-1}\,\eta^{p-2}\delta\, L^{p-1}\,.
\]

Choosing $L=C+1$ has the function $C-z$ taking on the value $-1$
when $z=L$, and taking $\eta$ sufficiently small ensures that the
polynomial \eqref{eq:Applied Example Polynomial Bound} takes on a
negative value when $z=L$. As that same polynomial takes on the positive
value $C$ when $z=0$, then it has a zero. As a consequence, then
$\rho$ is uniformly bounded.
\end{proof}

\section*{Acknowledgment}

The author would like to thank Mathew Johnson and Kevin Zumbrun for
useful discussion and feedback.

\bibliographystyle{plain}
\bibliography{3C__Users_Connor_Dropbox_School_Research_Nonlin___ear_Instability_Paper_Nonlinear_Instability,4C__Users_Connor_Dropbox_School_Research_Nonlin___Instability_Paper_Dissipative_Type_Examples,5C__Users_Connor_Dropbox_School_Research_Nonlin___lity_Paper_Reaction-Diffusion_Type_Examples}

\end{document}